\documentclass[12pt]{amsart}
\usepackage{color,graphicx,latexsym,amssymb,amsmath,amsfonts,amsthm,float}
\usepackage{url}
\usepackage{hyperref}
\usepackage{geometry}[1.5in]
\usepackage[T1]{fontenc}

\newtheorem{theorem}{Theorem}
\newtheorem{lemma}[theorem]{Lemma}

\theoremstyle{definition}
\newtheorem*{remark}{Remark}

\newcommand{\Z}{\mathbb Z}
\newcommand{\E}{\mathbb E}

\definecolor{db}{rgb}{0.1,0,0.75}
\definecolor{lm}{cmyk}{0 ,1,0,0}

\newcommand{\Q}{\mathbb Q}

\newenvironment{pfofthm}[1]
{\par\vskip2\parsep\noindent{\sc Proof of Theorem\ #1. }}{{\hfill
$\Box$}
\par\vskip2\parsep}

\newtheorem*{LinMeshWall}{Linial--Meshulam--Wallach theorem}

\newcommand{\prop}{\mathcal P}
\newcommand{\pr}{\mathbb P}
\newcommand{\N}{\mathbb N}
\newcommand{\old}[1]{}
\newcommand{\qrs}{\text{$q$-reducing set }}

\title[The threshold for integer homology]{The threshold for integer homology in random $d$-complexes}
\author[Hoffman]{Christopher Hoffman}
\address{University of Washington}
\email{hoffman@math.washington.edu}
\author[Kahle]{Matthew Kahle}
\address{The Ohio State University}
\email{mkahle@math.osu.edu}
\author[Paquette]{Elliot Paquette}
\address{Weizmann Institute of Science}
\email{paquette@weizmann.ac.il}

\date{\today}

\begin{document}

\maketitle

\begin{abstract}  Let $Y \sim Y_d(n,p)$ denote the Bernoulli random $d$-dimensional simplicial complex.
We answer a question of Linial and Meshulam from 2003, showing that the threshold for vanishing of homology $H_{d-1}(Y; \Z)$ is less than $80d \log n / n$. This bound is tight, up to a constant factor.

\end{abstract}

\section{Introduction}

Define $Y_d(n,p)$ to be the probability distribution on all $d$-dimensional simplicial complexes with $n$ vertices, with complete $(d-1)$-skeleton and with each $d$-dimensional face included independently with probability $p$.  We use the notation $Y \sim Y_d(n,p)$ to mean that $Y$ is chosen according to the distribution $Y_d(n,p)$; note the $1$-dimensional case $Y_1(n,p)$ is equivalent to the Erd\H{o}s--R\'enyi random graph $G \sim G(n,p)$.

Results in this area are usually as $n \to \infty$ and $p=p(n)$. We say that an event occurs {\it with high probability} (abbreviated w.h.p.) if the probability approaches one as the number of vertices $n \to \infty$. Whenever we use big-$O$ or little-$o$ notation, it is also understood as $n \to \infty$.

\medskip

A function $f=f(n)$ is said to be a {\it threshold} for a property $\mathcal{P}$ if whenever $p / f \to \infty$, w.h.p.\ $G \in \prop$, and whenever $p / f \to 0$, w.h.p.\ $G \notin \mathcal{P}$. In this case, one often writes that $f$ is {\it the} threshold, even though technically $f$ is only defined up to a scalar factor.
%
%There are more refined notions of threshold. If there is a function $g = o(f)$ such that whenever $ p \ge f + g$, w.h.p.\ $G \in \prop$, and whenever $p \le f - g$, w.h.p.\ $G \notin \prop$, then $f$ is said to be a {\it sharp} threshold for $\prop$.

It is a fundamental fact of random graph theory (see for example Section 1.5 of \cite{JLR}) that every monotone property has a threshold. However, not every monotone property has a sharp threshold. For example, $1/n$ is the threshold for the appearance of triangles in $G(n,p)$, but this threshold is not sharp. In contrast, the Erd\H{o}s--R\'enyi theorem asserts that $\log n / n$ is a sharp threshold for connectivity. Classifying which graph properties have sharp thresholds is a problem which has been extensively studied; see for example the paper of Friedgut with appendix by Bourgain \cite{FB99}.

\medskip

The first theorem concerning the topology of $Y_d(n,p)$ was in the influential paper of Linial and Meshulam \cite{LM}.
Their results were extended by Meshulam and Wallach to prove the following far reaching extension of the Erd\H{o}s--R\'enyi theorem \cite{MW09}, where they described sharp vanishing thresholds for homology with field coefficients.
\begin{LinMeshWall}
Suppose that $d \ge 2$ is fixed and that $Y \sim Y_d(n,p)$. Let $\omega$ be any function such that  $\omega \to \infty$ as $n \to \infty$.
\begin{enumerate}
\item If $$p \le  \frac{d  \log{n} - \omega }{ n}$$ then w.h.p.\ $H_{d-1}(Y; \Z / q\Z) \neq 0$, and
\item if $$p \ge \frac{d  \log{n}+\omega }{ n}$$ then w.h.p.\ $H_{d-1}(Y; \Z / q\Z) = 0$.
\label{difficult}
\end{enumerate}
\end{LinMeshWall}

The $d=1$ case is equivalent to the Erd\H{o}s--R\'enyi theorem. The Linial--Meshulam theorem is the case $d=2$, $q=2$, and the Meshulam--Wallach theorem is the general case $d \ge 2$ arbitrary and $q$ any fixed prime. In closing remarks of \cite{LM}, Linial and Meshulam asked ``Where is the threshold for the vanishing of $H_1(Y,\Z)$?''

By the universal coefficient theorem, $H_{d-1}(Y;\Z/q\Z)=0$ for every prime $q$ implies that $H_{d-1}(Y; \Z)=0$, so one may be tempted to conclude that the Meshulam--Wallach theorem already answers the question of the threshold for $\Z$-coefficients. This is not the case, however, since we are concerned with not just a single simplicial complex, but with a sequence of complexes as $n \to \infty$, and there might very well be torsion growing with $n$. The Meshulam--Wallach Theorem holds for $q$ fixed, and can be made to work for $q$ growing slowly enough compared with $n$. But it does not seem possible to extend the cocycle-counting arguments from \cite{LM} and \cite{MW09} to cover the case when $q$ is growing much faster than polynomial in $n$.

On the surface of things, this might actually be a big problem.  A complex $X$ is called $\Q$-acyclic if $H_0(X,\Q) = \Q$ and $H_i(X,\Q) = 0$ for $i \geq 1.$  Kalai showed that for a uniform random $\Q$-acyclic $2$-dimensional complex $T$ with $n$ vertices and ${n-1} \choose 2$ edges, the expected size of the torsion group $|H_1(T;\Z)|$ is of order at least $\exp(c n^2)$ for come constant $c > 0$ \cite{Kalai83}. On the other hand, the largest possible torsion for a $2$-complex on $n$ vertices is of order at most $\exp(C n^2)$ for some other constant $C > 0$, so Kalai's random $\Q$-acyclic complex provides a model of random simplicial complex which is essentially the worst case scenario for torsion.

We mention in passing that another approach to homology-vanishing theorems for random simplicial complexes is ``Garland's method'' \cite{Garland}, with various refinements due to \.Zuk \cite{Zuk, zuk2}, Ballman--\'Swi\k{a}tkowski \cite{BS99}, and others. These methods have been applied in the context of random simplicial complexes, see for example \cite{hkp12,kahle14}. However, it must be emphasized that these methods only work over a field of characteristic zero; they do not detect torsion in homology. A different kind of argument is needed to handle homology with $\Z$ coefficients.

\medskip

The fundamental group $\pi_1(Y)$ of the random $2$-complex $Y \sim Y_2(n,p)$ was studied earlier by Babson, Hoffman, and Kahle \cite{BHK11}, and the threshold face probability for simple connectivity was shown to be of order $1 / \sqrt{n}$. Until now, there seems to have been no upper bound on the vanishing threshold for integer homology for random $2$-complexes, other than this.

Our main result is that the threshold for vanishing of integral homology agrees with the threshold for field coefficients, up to a constant factor. In particular we have the following.

%\begin{theorem} 
%\label{thm:main1}
%Suppose that $Y \sim Y(n,p)$ and $\epsilon > 0$ is fixed. Then
%\begin{enumerate}
%%\item If $$p \le \frac{(2 - \epsilon)\log{n} }{ n}$$ then w.h.p.\ $H_1(Y, \Z ) \neq 0$, 
%\item if $p \ge \frac{25d \log{n}}{ n}$ then   for all primes $q$ $$\pr(H_1(Y; \Z/q\Z) \neq 0) \leq n^{-(d+1)}$$ and
%\item if $p \ge \frac{80d \log{n}}{ n}$ then  $$\pr(H_1(Y; \Z) \neq 0)=o(1).$$
%\label{difficult}
%\end{enumerate}
%\end{theorem}
%

\begin{theorem} 
\label{thm:main1}
Let $d \ge 2$ be fixed and $Y \sim Y_d(n,p)$. If
$$p \ge \frac{80d \log{n}}{ n}$$ then  $H_{d-1}(Y; \Z)  = 0$ w.h.p.
\end{theorem}

\begin{remark}
For the sake of simplicity, we make no attempt here to optimize the constant $80d$. We conjecture that the best possible constant is $d$; in other words we would guess that the Linial--Meshulam--Wallach theorem is still true with $\Z / q\Z$-coefficients replaced by $\Z$-coefficients. But to prove this, it seems that another idea will be required.
\end{remark}
%
%\begin{remark}
%Combined with \cite{FB99}, our result implies that the threshold function for the vanishing of integer homology is bounded away from an integer power of $n$ and thus vanishing of integer homology also has a sharp threshold.
%\end{remark}

%\begin{remark}
%Our proof of Theorem \ref{thm:main1} starts by assuming that we already know that $H_{d-1}(Y; \Q) = 0$. This follows from either the Linial--Meshulam--Wallach theorem, together with the universal coefficients theorem, or from the spectral gap results in the recent \cite{hkp12}, together with Garland's method \cite{Garland}. Theorem \ref{thm:main2} below also provides a self-contained and elementary proof of this fact.
%\end{remark}

\bigskip

Our main tool in proving Theorem \ref{thm:main1} is the following.

\begin{theorem}
\label{thm:main2}
Let $d \ge 2$ be fixed and let $q=q(n)$ be a sequence of primes. If $Y \sim Y_d(n,p)$ where  
$$p \ge \frac{40d \log{n}}{ n},$$
then
$$\pr(H_{d-1}(Y; \Z/q\Z) \neq 0) \leq \frac{1}{n^{d+1}}.$$
\end{theorem}

\begin{remark}
Theorem \ref{thm:main2} is similar to the main result in Meshulam--Wallach, but the statement and proof differ in fundamental ways. The main point is that the bound on the probability that $H_{d-1}(Y; \Z/q\Z) \neq 0$ holds uniformly over all primes $q$, even if $q$ is growing very quickly compared to the number of vertices $n$.
\end{remark}

\section{Proof}

%\begin{abstract}
%We study the homology with $\mathbb{Z}$ coefficients of a random 2 dimensional Linial Meshulam random simplicial complex.
%\end{abstract}

We first prove Theorem \ref{thm:main1}. The proof relies on 
Theorem \ref{thm:main2} plus one additional fact --- a bound on the size of the torsion subgroup in the degree $(d-1)$ homology of a simplicial complex, which only depends on the number of vertices $n$.  Let $A_{T}$ denote the torsion subgroup of an abelian group $A.$
\begin{lemma} \label{dino bath}
Let $d \ge 2$ and suppose that $X$ is a $d$-dimensional simplicial complex on $n$ vertices. 
Then $|\left(H_{d-1}(X;\Z)\right)_{T}| = \exp\left ( O(n^d) \right) $. 
\end{lemma}

\begin{proof}[Proof of Lemma \ref{dino bath}]
We include a proof here for the sake of completeness, but such bounds on the order of torsion groups are known. See, for example, Proposition 3 in Soul\'{e} \cite{Soule99}, which he attributes in turn to Gabber.

We assume without loss of generality that $H_{d}(X) =0$. Indeed, if there is a nontrivial cycle $Z$ in $H_{d}(X)$, then delete one face $\sigma$ from the support of $Z$. Then in the subcomplex $X - \sigma$, the rank of $H_{d}(X - \sigma)$ is one less than the rank of $H_{d}(X)$. So we have
$$\dim [ H_{d-1} (X - \sigma, {\bf k}) ] = \dim [ H_{d-1} (X , {\bf k} )]$$ over every field $k$, and then the isomorphism
$H_{d-1} (X - \sigma, \Z)  = H_{d-1} (X , \Z )$ follows by the universal coefficient theorem. 

We may further assume that the number of $d$-dimensional faces $f_d$ is bounded by $f_d \le {n \choose d}$, since if there were more faces than this, then we would have $f_d > f_{d-1}$ and there would have to be nontrivial homology in degree $d$, by dimensional considerations.

Let $C_i$ denote the space of chains in degree $i$, i.e.\ all formal $\Z$-linear combinations of $i$-dimensional faces, and let $\delta_{i}: C_i \to C_{i-1}$ be the boundary map in simplicial homology. If $Z_i$ is the kernel of $\delta_{i}$ and $B_i$ is the image of $\delta_{i+1}$, then by definition $H_i (X; \Z) = Z_i / B_i$.

Let $M_i$ be a matrix for the boundary map $\delta_i$, with respect to the preferred bases of faces in the simplicial complex. Then the order of the torsion subgroup $|(C_i / B_i)_T|$ is bounded by the product of the lengths of the columns of $M_i$, as follows.

%By changing bases for $C_i$ and $C_{i+1}$, we replace $M_i$ by its Smith normal form $D$, and $D$ only has nonzero entries along the main diagonal. By deleting some zero rows and columns of $D$, i.e.\ by restricting to appropriate subspaces of $C_{i}$ and $C_{i-1},$ we obtain a matrix $D'$ which is square and such that $|\det D'| = |(C_i / B_i)_T |$. The Hadamard bound gives that the product of the lengths of columns of $D'$ is an upper bound on $|\det D'|$. Moreover, the product of the lengths of the columns of $D'$ is no more than the product of the lengths of the nonzero columns of $M_i$.

We begin by writing $M_i$ in its Smith normal form, i.e. $M_i = P D Q$ with $P$ and $Q$ invertible matrices over $\Z$ and $D$ a rectangular matrix with entries only on its diagonal.  Let $r$ be the rank of $D$ over $\Q;$ note this is also the $\Q$-rank of $M_i.$  By removing the all $0$ rows and columns from $D$ (and some columns of $P$ and some rows of $Q$), we may write $M_i = P' D' Q'$ where $D'$ is an $r \times r$ diagonal matrix, and all of $P',D',\text{and } Q'$ have $\Q$-rank $r.$  By the definition of $D,$ we have $\det D' = |(C_i / B_i)_T|.$

As $P'$ and $Q'$ both have $\Q$-rank $r,$ we can find a collection of $r$ rows from $P'$ that are linearly independent over $\Q$ and $r$ columns of $Q'$ that are linearly independent over $\Q.$  Write $\tilde P$ and $\tilde Q$ for the $r \times r$ submatrices of $P'$ and $Q'$ given by these rows and columns.  As $\tilde P$ and $\tilde Q$ are full $\Q$-rank, they are invertible over $\Q$ and have nonzero determinant.  As they are additionally integer matrices, they each have determinants at least $1.$ Thus,
\[
 \det(D') 
 \leq | \det( \tilde P) \det(D') \det( \tilde Q)| 
 = | \det( \tilde PD'\tilde Q)|.
\]

On the other hand $\tilde M = \tilde P D' \tilde Q$ is an $r \times r$ submatrix of $M_i.$  Thus, applying the Hadamard bound to $\tilde M$, we may bound $\det(\tilde M)$ by the product of the lengths of the columns of $\tilde M.$  As the columns of $M_i$ all have lengths at least $1,$ the product of the lengths of the columns of $\tilde M$ are at most the product of the lengths of the columns of $M_i,$ completing the proof.

Since $Z_i / B_i$ is isomorphic to a subgroup of $C_i / B_i$, this also gives a bound on the torsion in homology. In particular, for any simplicial complex $X$ on $n$ vertices, we have that
\begin{align*}
|\left(H_{d-1}(X;\Z)\right)_{T}| & \le \sqrt{d+1}^{{ n \choose d  }} \\
& = \exp\left ( O(n^d) \right).
\end{align*}

%%03-18-14 Edit start
%Let $M_i$ be a matrix for the boundary map $\delta_i$, with respect to the preferred bases of faces in the simplicial complex. Then the order of the torsion subgroup $|(Z_i / B_i)_T|$ is bounded by the product of the lengths of the columns of $M_i$, as follows.

%By changing bases for $C_i$ and $C_{i-1}$, we replace $M_i$ by its Smith normal form $D$, and $D$ only has nonzero entries along the main diagonal. By deleting some zero rows and columns of $D$, i.e.\ by restricting to appropriate subspaces of $C_{i}$ and $C_{i-1},$ we obtain a matrix $D'$ which is square and such that $|\det D'| = |(Z_i / B_i)_T |$. The Hadamard bound gives that the product of the lengths of columns of $D'$ is an upper bound on $|\det D'|$. Moreover, the product of the lengths of the columns of $D'$ is no more than the product of the lengths of the nonzero columns of $M_i$.

%%Since $Z_i / B_i$ is isomorphic to a subgroup of $C_i / B_i$, this also gives a bound on the torsion in homology. 
%In particular, for any simplicial complex $X$ on $n$ vertices, we have that
%%03-18-14 Edit end

\end{proof}

Now define
$$Q(X)=\{q~\text{prime}:\ H_{d-1}(X;\Z/q\Z) \neq 0\}.$$ 
An immediate consequence of Lemma \ref{dino bath} is that $$|Q(X)| =O( n^d),$$
and this is the fact which we will use.

\medskip

\begin{pfofthm}{\ref{thm:main1}}
Our strategy is as follows. 
%Let $C(d)$ be a constant that will be specified later. 
Let $Y_1,Y_2 \sim Y_d(n,40d \log n/n)$ 
be two independent random $d$-complexes 
and let $Y \sim Y_d(n,80d \log n/n)$
\begin{enumerate}
\item[\bf{Step 1}] \label{duwamish} First we note that we can couple $Y$, $Y_1$ and $Y_2$ such that 
\begin{equation} F_d(Y_1) \cup F_d(Y_2)\subset F_d(Y) .  \label{cauliflower} \end{equation}
By (\ref{cauliflower}) if $H_{d-1}(Y_1;\Z/q\Z)=0$ or $H_{d-1}(Y_2;\Z/q\Z)=0$ then 
$H_{d-1}(Y;\Z/q\Z)=0$.
%\item The bulk of our work will be to give a simple proof of a result similar to Meshulam--Wallach.  In particular we prove that for all $q$ $$\P(H_{d-1}(Y_2,\Z/q\Z) \neq 0)\leq n^{-d-1}.$$ 

\item[\bf{Step 2}] 
  %By Theorem \ref{thm:main2}, with probability at least $1-n^{-(d+1)},$ 
%the random complex $Y_1$ has trivial $d-1$st homology and 
%\begin{equation}\label{ramrod} H_{d-1}(Y_1;\Z/2\Z)=0. \end{equation}
%As $H_{d-2}(Y_1;\Z) = 0$ by virtue of containing the complete $(d-1)$-skeleton, the universal
%coefficient theorem gives
%\(
%0=H_{d-1}(Y_1;\Z/2\Z)=H_{d-1}(Y_1;\Z) \otimes \Z/2\Z,
%\)
%so that $H_{d-1}(Y_1;\Z)$ is all torsion.
%in which case $H_{d-1}(Y_1;\Q)=0$ by the universal coefficient theorem. 
%Let $Q'$ be the set of $q$ such that $H_{d-1}(Y_1,\Z/q\Z) \neq 0$. 
%{\color{red} We don't need anything on the homology of $Y_1,$ as all lower $\Z$-homologies vanish, to conclude the following statement.}
%By Lemma \ref{dino bath}, if (\ref{ramrod}) is satisfied then $Q(Y_1)$ has cardinality $O(n^d)$.
By Lemma \ref{dino bath}, $Q(Y_1)$ has cardinality $O(n^d)$.
\item[\bf{Step 3}] Applying a union bound, the probability that either $H_{d-1}(Y_1;\Q) \neq 0$ or there exists $q \in Q(Y_1)$ such that 
$$H_{d-1}(Y_2;\Z/q\Z) \neq 0$$
is at most $O(n^d\cdot n^{-(d+1)})=O(1/n)=o(1)$.
\item[\bf{Step 4}] Thus if 
\begin{enumerate}
  \item $H_{d-1}(Y_1;\Q) =0$, and
  \item $H_{d-1}(Y_2;\Z/q\Z) = 0$ for all 
$q \in Q(Y_1)$,
\end{enumerate}
then by the coupling in Step 1, we have that $H_{d-1}(Y;\Z/q\Z)=0$ for all primes $q$. By the universal coefficient theorem we have that $H_{d-1}(Y;\Z)=0$. Each of these two conditions happens with probability  $1-o(1)$ which completes the proof.
\end{enumerate}
\end{pfofthm}

Now we begin our proof of Theorem \ref{thm:main2}. Throughout this paper we are always working with $d$-dimensional simplicial complexes on vertex set $[n]$, with complete $(d-1)$-skeleton. Such a complex $Y$ is defined by $F_d(Y)$, its set of $d$-dimensional faces. We often associate the two in the following way. If $f\in {[n] \choose d+1}$ (i.e.\ $f$ is a $d$-dimensional simplex) and $Y$ is as above then we write $Y \cup f$ for the simplicial complex with 
$F_d(Y \cup f)=F_d(Y)\cup f.$
 
\bigskip
 
Let $q$ be a prime and $Y$ be as above. Define
$$\qrs(Y)=\left \{f :\ H_{d-1}(Y\cup f; \Z/q\Z) \neq H_{d-1}(Y; \Z/q\Z)\right \}.$$

In other words, $\qrs(f)$ is precisely the set of $d$-dimensional faces which, when added to $Y$, drop the dimension of $H_{d-1} (Y; \Z/ q \Z)$ by one.

\begin{lemma} \label{shabby}
A $d$-dimensional simplex $f \in \qrs(Y)$ if and only if the boundary of $f$ is not a $(\Z/q\Z)$ boundary in $Y$. Thus if $Y \subset Y'$, where $Y$ and $Y'$ are $d$-dimensional complexes sharing the same $d-1$-skeleton,
then $$\qrs(Y') \subset \qrs(Y).$$
\end{lemma}

\begin{proof}
%Let $c$ be the cycle of length 3 which is the boundary of $f$. 
If $\partial f$ is not a boundary in $Y$ then 
$H_{d-1}(Y;Z/q\Z) \neq H_{d-1}(Y \cup f;Z/q\Z)$. If $\partial f$ is a boundary in $Y$ then 
$H_{d-1}(Y;Z/q\Z) = H_{d-1}(Y \cup f;Z/q\Z)$.
 \end{proof}

\begin{lemma} \label{djimi}
$H_{d-1}(Y; \Z/q\Z) =0$ if and only if $\qrs(Y)=\emptyset.$
\end{lemma}

\begin{proof}
Clearly, $H_{d-1}(*,\Z/q\Z)=0$  is monotone with respect to inclusion of $d$-faces, so $H_{d-1}(Y; \Z/q\Z) =0$ implies that $\qrs(Y)=\emptyset.$

But we also have that the $d-1$-skeleton of $Y$ is complete, so once all possible $d$-faces have been added, homology is vanishing. Once again applying the monotonicity of Lemma \ref{shabby} , $\qrs(Y) = \emptyset$ also implies that $H_{d-1}(Y; \Z/ q \Z) = 0$.

\end{proof}

Instead of working directly with the Linial--Meshulam distribution $Y_d(n,p)$ where each face is included independently with probability $p$, it is convenient to work with the closely related distribution $Y_d(n,m)$, where the complex is chosen uniformly over all 
$${ {n \choose d+1} \choose m}$$
simplicial complexes on $n$ vertices with complete $d-1$-skeleton, and with exactly $m$ faces of dimension $d$. As with the random graphs we have that if $m \approx p{n \choose d+1}$ then for many properties the two models are very similar. After doing our analysis with $Y_d(n,m)$, we convert our results back to the case of $Y_d(n,p)$.

Let 
$$ \tilde m =  \tilde m(n,q)= \min\left\{m':\ \E\big| \qrs(Y(n,m'))\big| \leq 
  \frac{1}{2}\binom{n}{d+1} \right\}
  $$
This next lemma points out an easy consequence of our definition of $\tilde m$.
\begin{lemma} \label{lamar}
For every $d$-face $f$
$$\pr \left( f \in \qrs \left( Y(n,\tilde m) \right) \right)\leq 1/2.$$
\end{lemma}
\begin{proof}
This follows easily by symmetry.
%By symmetry $$\pr\big(f \in \qrs(Y(n,m'))\big)$$ is independent of $f$. Thus for any $n$, $m'$ and $f$
%$$\E\big| \qrs(Y(n,m'))\big| =\frac{n^{d+1}}{d!}\pr(f \in \qrs(Y(n,m')).$$
%Thus
%$$ \tilde m  = \min\{m':\ \pr(f \in \qrs(Y(n,m')))\leq 1/2\}$$
%and
%\begin{equation} %\label{lamar}
%\pr(f \in \qrs(Y(n,\tilde m)))\leq 1/2.
%\end{equation}
\end{proof}

If $Z$ and $Z'$ are random $d$-complexes with vertex set $[n]$ and the complete $(d-1)$-skeleton
then we say $Z$ {\it stochastically dominates} $Z'$ if there exists a coupling of the two random variables with $\pr\bigg(F_d(Z') \subset F_d(Z)\bigg)=1$.

\begin{lemma} \label{whitecaps}
Let $m=\sum_{i=1}^k m_i$ with $m_i \in \N$.
Also let $Y \sim Y_d(n,m)$ and $Y^i \sim Y_d(n,m_i)$ for all $i$.
Then  $Y$ stochastically dominates $\bigcup_{i=1}^{k}Y^i$ and
$$\qrs(Y) \subset \qrs\left( \bigcup_{i=1}^{k}Y^i \right).$$
\end{lemma}

\begin{proof}
The first claim is a standard argument; see for example Section 1.1 of \cite{JLR}. The second follows from the first and the monotonicity of the $q$-reducing set (Lemma \ref{shabby}).
\end{proof}

\begin{lemma} \label{traore}
For any $q$, sufficiently large $n$, $d$-face $f$ and $k\geq 2(d+1) \log_2(n)$, then
for $Y \sim Y_d(n,k \tilde m)$
$$\pr\bigg(f \in \qrs(Y)\bigg)\leq \frac{1}{n^{2(d+1)}}.$$
\end{lemma}

\begin{proof}
Let $Y^1, \dots , Y^k$ be i.i.d.\ complexes with distribution $Y_d(n,\tilde m)$.
Then by Lemma \ref{whitecaps}
%$$Y(n,k\tilde m) \text{ stochastically dominates } \bigcup_{i=1}^{k}Y^i$$
%and 
we can find a coupling so that a.s.\ 
$$\qrs\big(Y\big) \subset \qrs\left(\bigcup_{i=1}^{k}Y^i\right). $$
Then by Lemmas \ref{shabby}, \ref{djimi} and  \ref{lamar}
\begin{eqnarray*}
\pr\bigg(f \in \qrs(Y) \bigg)
& \leq & \pr\left( f \in \qrs\left( \bigcup_{1}^{k}Y^i \right) \right) \\
& \leq & \pr\left(\bigcap_1^k \left\{f \in \qrs\left( Y^i\right)\right\} \right) \\
& \leq & \prod_1^k \pr \left( f \in \qrs\left( Y^i\right)\right)  \\
& \leq & \left( \frac12 \right)^k\\
& \leq & \frac{1}{n^{2(d+1)}}.
\end{eqnarray*}
\end{proof}

Now the main task that remains is to estimate $\tilde m$. Before we do so, we give a heuristic that indicates that
$\tilde m \leq 2{n \choose d}$. We consider the process where we start with $Y_0$ the complex with the complete $(d-1)$-skeleton and no $d$-dimensional faces. Then we inductively generate $Y_{i+1}$ by taking $Y_i$ and independently adding one new $d$-dimensional face.  Note that when we are adding faces one at a time, the dimension $\dim H_{d-1}(Y_i, \Z / q \Z)$ is monotone decreasing.
 
As $H_{d-1}(Y_0;\Z/q\Z)$ is generated by the $(d-1)$-cycles its dimension is at most ${n \choose d}$. 
%So at most ${n \choose d}$ of the faces can reduce the dimension of homology. 
Heuristically this indicates that $\tilde m$ should be no larger than $2{n \choose d}$, because 
if we were to add $2{n \choose d}$ faces and half of them reduce the dimension of the 
homology, then the dimension has dropped 
${n \choose d}$ times.  This would make the homology trivial, and would leave no faces remaining in the $q$-reducing set.  We now make this heuristic rigorous, albeit with a slightly worse constant.

%\begin{lemma}
%
%\end{lemma}
%
%\begin{proof}
%\end{proof}

\begin{lemma} \label{froome}
Let $Y$ be a $d$-complex and let $f_1, f_2, \dots$ be an ordering of $F_d(Y)$.
Let $Y_i$ be the $d$-complex with $$F_d(Y_i)=\bigcup_{j=1}^{i}\{f_j\}.$$ Then there are at most ${n \choose d}$ $i$ such that
$$f_i \in \qrs(Y_{i-1}). $$
\end{lemma}

\begin{proof}
By induction. If there exist a subsequence $0<i_1<i_2<\dots <i_{s}$ with
$$f_{i_s}\in \qrs(Y_{i_s-1})$$
then 
$$|H_{d-1}(Y_{i_s},\Z/q\Z)| \leq q^{{n \choose d}-s}.$$
Thus the longest possible subsequence has length ${n \choose d}$.
\end{proof}

\begin{lemma} \label{spokane}
For any $q$ and any  $n>d$ we have $\tilde m \leq 4{n \choose d}$.
\end{lemma}

\begin{proof}
%Let $f_1,f_2, \dots, f_{{n \choose d+1}}$ be a uniform ordering of the possible $d$-faces.
Let $f_1,f_2, \dots, f_{{n \choose d+1}}$ be a uniformly random ordering of all the possible $d$-faces.
%For every $l$ define $Y_l$ to be the 2-complex with
%$$F_2(Y_l)=\bigcup_{i=1}^lf_i.$$
Again we define the complexes $Y_i$ by $$F_d(Y_i)=\bigcup_{j=1}^{i}\{f_j\},$$
and we remark that each $F_d(Y_i)$ is distributed as $Y_d(n,m).$
Define the random variables $$Z_i={\bf 1}_{\{f_i \in \qrs(Y_{i-1})\}}.$$
and $\{X_i\}$ be an i.i.d.\ sequence of Bernoulli(1/3) random variables.
We can couple the events 
so that $Z_i$ stochastically dominates $X_i$ up until the random time $m^*$,
where
$$m^*=\min\left(m':\ |\qrs(Y_{m'})| \leq 
\frac{1}{3}\binom{n}{d+1}
\right).$$
Thus by Lemma \ref{froome} we have a.s.\ that
$${n \choose d} \geq \sum_{i=1}^{m^*}Z_i \geq \sum_{i=1}^{m^*}X_i.$$
So either
\begin{enumerate}
\item $m^* \leq 4{n \choose d}$ or 
\item $\sum_{i=1}^{4{n \choose d}}X_i <  {n \choose d}$. \label{brunch}
\end{enumerate}
The sum on the left hand side of \ref{brunch} has expected value $\frac43{n \choose d}$ which is a constant factor larger than ${n \choose d}.$ Thus the probability of the last event is exponentially decreasing in ${n \choose d},$ and so it is certainly less than 1/10.
Thus $\pr(m^* > 4{n \choose d})<1/10$ as well.

\begin{align*}
\E \big| \qrs\left(Y_{4{n \choose d}}\right)\big|
&\leq \frac{1}{3}\binom{n}{d+1}\cdot \pr\left(m^* \leq 4{n \choose d}\right)\\
&\hspace{0.5in} + \binom{n}{d+1}\pr\left(m^*>4{n \choose d}\right)\\
&\leq \frac{1}{3}\binom{n}{d+1} +\frac{1}{10}\binom{n}{d+1}\\
&\leq \frac{1}{2}\binom{n}{d+1}.
\end{align*}
Thus $\tilde m \leq 4 {n \choose d}.$
\end{proof}

\begin{lemma} \label{cotton}
Let $Y \sim Y_d(n,m)$. 
%For any $$m \geq (12d+12)(\log n){n \choose d}\geq (8d+8)(\log_2 n){n \choose d}$$
For $$m \geq (12d+12)(\log n) {n \choose d}.$$
 and any prime $q$
$$\pr\bigg(H_{d-1}(Y;\Z/q\Z)\neq 0\bigg)\leq\frac{1}{2n^{d+1}}. $$
\end{lemma}

\begin{proof}
First of all, $(12d+12)(\log n) > (8d+8)(\log_2 n)$, since $\log 2 > 2/3$.
Then by Lemma \ref{spokane} 
\begin{align*}
(8d+8)(\log_2 n){n \choose d} &=(2d+2)(\log_2 n)\left(4{n \choose d}\right) \\
& \geq (2d+2)(\log_2 n) \tilde m.\\
\end{align*}

By Lemma \ref{traore} and the union bound, we have 
\begin{align*}
\pr\bigg(H_{d-1}(Y;\Z/q\Z)\neq 0\bigg) & \leq {n \choose d+1}\frac{1}{n^{2(d+1)}}\\
& \leq \frac{1}{2n^{d+1}}.
\end{align*}
\end{proof}

\begin{pfofthm}{\ref{thm:main2}} 
%We can choose $C=C(d)$ so that w
If $$p \ge \frac{40d \log n}{n},$$ then by applying Chernoff bounds, with probability at least $$1-\frac{1}{2n^{d+1}}$$
a random $d$-complex $Y \sim Y_d(n,p)$ has at least $(12 d + 12)(\log n){n \choose d}$ faces of dimension $d$. 
Then the theorem follows from Lemma \ref{cotton}.
\end{pfofthm}

\section*{Acknowledgements}

The authors thank Nati Linial and Roy Meshulam for many helpful and encouraging conversations.

\medskip

C.H.\ gratefully acknowledges support from NSF grant DMS-1308645 and NSA grant H98230-13-1-0827. M.K.\ gratefully acknowledges support from the Alfred P.\ Sloan Foundation, from DARPA grant N66001-12-1-4226, and from NSF grant CCF-1017182. E.P.\ gratefully acknowledges support from NSF grant DMS-0847661.

\bibliographystyle{plain}
\bibliography{homrefs}

\end{document}